\documentclass[nomarks,
final
]{dmtcs-episciences}

\usepackage[utf8]{inputenc}
\usepackage{subfigure}

\usepackage{amsmath,amsthm,amssymb,amsfonts}

\usepackage{url}

\usepackage{color}
\usepackage{textcomp}
\usepackage{siunitx}
\usepackage{multirow}
\usepackage{float}
\usepackage{pdflscape}
\usepackage{afterpage}
\usepackage{cleveref}
\usepackage[toc,page,titletoc,title]{appendix}

\newtheorem{theorem}{Theorem}[section]
\theoremstyle{plain}

\newtheorem{definition}[theorem]{Definition}
\newtheorem{construction}[theorem]{Construction}

\newtheorem{proposition}[theorem]{Proposition}

\newtheorem{observation}[theorem]{Observation}
\newtheorem{question}[theorem]{Question}

\author{Gerold J\"ager 
	\and Klas Markstr\"om 
	\and Denys Shcherbak 	\and Lars-Daniel \"Ohman 	
	}
\title{Small Youden Rectangles, Near Youden Rectangles, and Their Connections to Other Row-Column Designs}
\affiliation{Department of Mathematics and Mathematical Statistics, Ume\aa{\ }University, Sweden}


\begin{document}

\publicationdata{vol. 25:1}{2023}{9}{10.46298/dmtcs.6754}{2020-09-02; 2020-09-02; 2022-05-16; 2022-11-01}{2023-01-05}

%----------------------------------------------------------------------
%----------------------------------------------------------------------

%----------------------------------------------------------------------
%----------------------------------------------------------------------

\maketitle

\begin{abstract}
	In this paper we first study $k \times n$ Youden rectangles of small 
	orders. We have enumerated all Youden rectangles for a range of small 
	parameter values, excluding the almost square cases where 
	$k = n-1$, in a large scale computer search. In particular, we
	verify the previous counts for $(n,k) = (7,3), (7,4)$, 
	and extend this to the cases $(11,5), (11,6), (13,4)$ and $(21,5)$.
	
	For small parameter values where no Youden rectangles exist, we 
	also enumerate rectangles where the number of symbols common to 
	two columns is always one of two possible values, differing by 1,
	which we call \emph{near Youden rectangles}.
	
	For all the designs we generate, we calculate the order of the 
	autotopism group and investigate to which degree a certain 
	transformation can yield other row-column 
	designs, namely double arrays, triple arrays and sesqui arrays.
	
	Finally, we also investigate certain Latin rectangles with three 
	possible pairwise intersection sizes for the columns and 
	demonstrate that these can give rise to triple and sesqui arrays 
	which cannot be obtained from Youden rectangles, using the 
	transformation mentioned above.
\end{abstract}

%----------------------------------------------------------------------
%----------------------------------------------------------------------
\section{Introduction}
%----------------------------------------------------------------------
%----------------------------------------------------------------------

An $(n,k,\lambda)$ Youden rectangle (sometimes referred to as a 
Youden square) where $n \geq k$ is a $k \times n$ array on $n$ symbols 
that satisfies the following two conditions: 
\begin{enumerate}
	\item There is no repeated symbol in any row or column, which 
		we will call the \emph{Latin condition}. 
	\item The number of shared symbols between any two columns is 
		always $\lambda$, which we will call the \emph{balance condition}. 
\end{enumerate}

Youden rectangles can be represented in different ways. In particular,
by switching the roles of rows and symbols, one gets a representation 
in the form of a square matrix, typically with some empty cells. In 
previous literature, the term `Youden square' has sometimes been used 
for the rectangular format as well, but we shall use the term 
`Youden rectangle' for the rectangular format, reserving the term 
`square' for the actual square format.

As indicated by the choice of terminology in the first part of the 
definition, a Youden rectangle can be viewed as a special case of a 
$k \times n$ Latin rectangle, which in this setting can be defined as 
a $k \times n$ array on $n$ symbols, satisfying the Latin condition. 
In the present paper, we exclude the square, and almost square cases 
$k=n$, $k=n-1$ as well as $k=1$ for Youden rectangles, since for these 
parameter choices, all Latin rectangles trivially also satisfy the 
second condition.

Clearly, each row will contain each symbol exactly once, and so
the array will also be \emph{equireplicate}, that is, each symbol
appears the same number of times, namely $k$. As is well known, 
divisibility and double counting considerations easily give that in 
order for a Youden rectangle to exist, $\lambda = \frac{k(k-1)}{n-1}$ 
must be an integer.

The reason for the use of the term `balance', is that when treating the 
columns of a Youden rectangle as sets of symbols, these sets form the 
blocks of a \emph{symmetric balanced incomplete block design} (SBIBD). 
Conversely, it was proven by Smith and Hartley~\cite{HS48} that the 
elements in the blocks of any SBIBD can be ordered to give a Youden 
rectangle. In fact, many different orderings are possible, so a single 
SBIBD will give rise to many different Youden rectangles. We have not
employed this connection between SBIBDs and Youden rectangles in our
computational work.

Alternatively, and equivalently, a Youden rectangle may be defined with
more of an SBIBD approach as a $k \times n$ array on $n$ symbols, where 
no symbol is repeated in any row, and when
viewing the columns as sets of symbols, each pair of symbols occurs
the same number of times, namely $\lambda$. The property that all
pairs of blocks in an SBIBD intersect in the same number of elements
is sometimes expressed by saying that the design is \emph{linked}, and 
more recently rather by saying that the design has 
\emph{constant block intersections}.

Already in the original paper \cite{You72} Youden points out that from 
a statistical point of view Youden rectangles suffer 
from the restricted set of feasible parameters. As one way around this 
problem we here introduce the class of \emph{near Youden rectangles}. 
For given values of $n$ and $k$ a near Youden rectangle is a Latin 
rectangle with two allowed block intersection sizes, differing by 1, 
rather than one single intersection size. This relaxation significantly 
increases the set of allowed parameters while in a sense still keeping 
the design as balanced as possible. In Section \ref{nyrintro} we discuss 
the theoretical properties of these designs in greater detail, and discuss 
their connections to existing design classes.

The early history of the study of Youden rectangles was chronicled by
Preece~\cite{Pr4}, and a good starting point for further reading is the 
Youden chapter in the Handbook of Combinatorial Designs~\cite{Handbook}.

Little has been done on complete enumeration of these objects, though 
in \cite{Preece66} Youden rectangles with $n\leq 7$ were classified by 
Preece, and in \cite{triples} we performed a full enumeration of mutually 
orthogonal (in the Latin rectangle sense) triples of Youden 
rectangles for $n\leq 7$. Note that orthogonal Youden 
rectangles should not be confused with \emph{multi-layered Youden 
rectangles}, as studied in \cite{Pr3}. In the present paper, our main 
aim has been to perform a 
complete enumeration of Youden rectangles for as large parameters as 
possible. The current state of knowledge 
on the number of Youden rectangles is tabulated in \cite{Handbook},
which goes up to $n=7$. 

The rest of the paper is structured as follows. 
In Section~\ref{sec_not} we give 
some further basic notation and formal definitions. In 
Section~\ref{sec_problem} we state the questions guiding our 
investigation, and describe briefly the method and algorithms used 
together with some practical information regarding the computer 
calculations. In Section~\ref{sec_results} we present the data 
our computer search resulted in, in particular the number of different 
Youden rectangles of some small orders. In Section~\ref{sec_TA}, we 
analyze the constructed objects with regards to other types of 
row-column designs. Section~\ref{sec_concl} concludes.

%----------------------------------------------------------------------
%----------------------------------------------------------------------
\section{Preliminaries}
\label{sec_not}
%----------------------------------------------------------------------
%----------------------------------------------------------------------%----------------------------------------------------------------------

%----------------------------------------------------------------------
\subsection{Notions of Equivalence}
%----------------------------------------------------------------------

We will use $\{0,1,\dots,n-1\}$ as the symbol set.
We call a Youden rectangle \textit{normalized} 
if it satisfies the following conditions:

\begin{description}
	\item[(S1)] (Ordering among columns) 
		The first row is the identity permutation.
	\item[(S2)] (Ordering among rows)
		The first column is $0,1,2,\ldots,k-1$.	
\end{description}

Two Youden rectangles $Y_A$ and $Y_B$ are said 
to be \emph{isotopic} if there exists a permutation $\pi_s$ of the 
symbols, a permutation $\pi_r$ of the rows and a permutation $\pi_c$ of 
the columns
such that when applying all three permutations to $Y_A$, we get $Y_B$.
The equivalence concept isotopism is perhaps the most natural one 
when studying Youden rectangles, and isotopism classes are also known 
as \emph{transformation sets} in this context.

Two normalized Youden rectangles $Y_A$ and $Y_B$ can be 
isotopic to each other, so grouping Youden rectangles according
to which normalized rectangle they yield when renaming the symbols
in the first column $0, 1, \ldots k-1$ in this order, and permuting
the columns to satisfy \textbf{S1} gives a weaker notion of equi\-valence, 
by saying that $Y_A$ and $Y_B$ are equivalent if they yield the same 
normalized Youden rectangle in this way.

Other concepts of 
equivalence are also possible, and allowing for exchanging the roles 
of symbols and columns leads to the notion of \emph{species} (also
known as \emph{main classes}). In the present paper, we will not be
employing the last mentioned notion of equivalence, and we comment on 
this choice below. Taking transposes (that is, exchanging 
the roles of columns and rows), or exchanging the roles of symbols and
rows, however, does not map $k \times n$ Youden rectangles to 
$k \times n$ Youden rectangles, and so we do not consider these 
transformations here.

Making this more formal, the group $G_{n,k}=S_k\times S_n \times S_n$ 
of isotopisms acts on the set of $k\times n$ Youden rectangles, where 
$S_k$ corresponds to a permutation of the rows, the first $S_n$ 
corresponds to a permutation 
of the columns, and the last $S_n$ corresponds to a permutation of the 
symbols. Two rectangles $Y_A$ and $Y_B$ of size $k\times n$ are 
isotopic, and we say that they belong to the same \emph{isotopism class} 
if there exists a $g\in G_{n,k}$ such that  $g(Y_A)=Y_B$. 
The \emph{autotopism group} of a Youden rectangle $Y$ is defined as 
$\mbox{Aut}(Y):=\{g\in G_{n,k}\, |\, g(Y)=Y\}$.
When presenting examples, we use normalized representatives of 
isotopism classes. For a recent survey on the concept of isotopism in 
algebra and designs, see~\cite{FFN}.

%----------------------------------------------------------------------
\subsection{Near Youden Rectangles}\label{nyrintro}
%----------------------------------------------------------------------

For parameters where $\lambda$ as calculated by 
$\lambda = \frac{k(k-1)}{n-1}$ is not an integer, no Youden rectangle
exists. This divisibility is quite restrictive and from a statistical 
design theory perspective it is desirable to include more parameter 
choices here. One natural relaxation is to allow two different column 
intersection sizes, leading us to the following definition:

\begin{definition}\label{def_PBY}
	A \emph{near Youden rectangle} (NYR) is a $k \times n$ 
	Latin rectangle where every column-column intersection has size 
	either $\lambda_1=\lfloor \lambda \rfloor$, or 
	$\lambda_2=\lceil \lambda \rceil$, 
	where $\lambda=\frac{k(k-1)}{n-1}$. 
\end{definition}

An example of a $4 \times 6$ NYR with column intersection sizes 
$\lambda_1 = 2$ and $\lambda_2 = 3$ is given in Figure~\ref{nYA6}. For 
example, the first column intersects the second, third and fourth columns 
in $2$ symbols, and the remaining columns in $3$ symbols.

\begin{figure}[h]
	\begin{center}
		\begin{tabular}{|cccccc|}
		\hline 			
		0 & 1 & 2 & 3 & 4 & 5 \\ 
		1 & 0 & 5 & 4 & 3 & 2 \\ 
		2 & 4 & 0 & 5 & 1 & 3 \\ 
		3 & 5 & 4 & 0 & 2 & 1 \\
		\hline 
		\end{tabular} 
	\end{center}
	\caption{A $4\times 6$ near Youden rectangle.}
	\label{nYA6}
\end{figure}

If $\lambda_1$ is zero, the resulting designs (when interpreting columns
as blocks) may be \emph{disconnected}, that is, the symbol set can be
partitioned in two parts $S_1$ and $S_2$ such that the set of columns where
the symbols in $S_1$ appear is disjoint from the set of columns where the 
symbols in $S_2$ appear. For example, two $7 \times 3$ Youden rectangles on
disjoint symbol sets may be juxtaposed to form a $14 \times 3$ near Youden 
rectangle. Disconnectedness is undesirable from a statistical design point,
but when $\lambda_1 \geq 1$, all near Youden rectangles are connected.

If we disregard the order of the elements in the columns of an NYR we get 
an equireplicate block design  with the same intersection property as 
the NYR, i.e. pairs of blocks intersect in either $\lambda_1$ or $\lambda_2$ 
elements. However, in the study of block designs it has been more 
common to define designs in terms of covering numbers for pairs of symbols, 
i.e. the number of blocks which contain the pair of symbols, rather 
than intersection numbers.  However, following Fisher's original proof of 
Fishers inequality in~\cite{MR2086}, rather than the now more common 
linear algebraic 
version, one can easily connect intersection numbers and covering numbers. 
The idea behind Fisher's proof is to calculate the variance of the 
intersection numbers in terms of the covering numbers, and as a corollary 
he also gets the result that in an SBIBD the intersection number is 
constant. This argument can also be done in the other direction, describing 
the variance of the covering numbers in terms of the intersection numbers.  
Instead of doing this from scratch we will use an identity given by Tsuji 
in \cite{Tsu}, though we note that similar identities were used earlier in 
\cite{Brown88}. We here state the identity in a less general form, adapted 
to our current situation. 

\begin{theorem}[Lemma 1 in \cite{Tsu}]\label{TsuThm}
	Let $l_{p,q}$ denote the number of columns which contain the pair 
	$\{p,q\}$ and $m_{i,j}$  the size of the intersection of the $i$:th 
	and $j$:th columns. With $\lambda$ as already defined, we then have
	\begin{equation*}
	\begin{aligned}
	\sum_{p,q} (l_{p,q} - \lambda)^2 = {} &
	\sum_{i,j}\left( m^2_{i,j} 
		- \left (1+2\frac{(k-1)^2}{n-2}\right)m_{i,j}\right.
	\\
	& + \left.\frac{k(k-1)}{n-1} 
		\left(1-2\frac{k-1}{n-2}+\frac{nk(k-1)}{(n-1)(n-2)}\right)\right)
	\end{aligned}
	\end{equation*}
	where the first sum is over 2-subsets of symbols and the second is 
	over 2-subsets of columns.  
\end{theorem}

Next we note that we can determine the number of column pairs with a given 
intersection size in a NYR, and that these intersection sizes are nicely 
distributed.

\begin{proposition}\label{prop_nYR_col_intersect}
	Let $A$ be a $k \times n$ near Youden rectangle with column 
	intersection sizes 
	$\lambda_1 = \left \lfloor \frac{k(k-1)}{n-1} \right \rfloor$ 
	and 
	$\lambda_2 = \left \lceil \frac{k(k-1)}{n-1} \right \rceil$. 
	Then any column $c$ intersects $n_1=\lambda_2(n-1)-k(k-1)$ other 
	columns in $\lambda_1$ symbols and 
	$n_2=- \lambda_1(n-1) + k(k-1)$ other  
	columns in $\lambda_2$ symbols.	
\end{proposition}

\begin{proof}
	We fix an arbitrary column $c$ and count the sum total $S$ of
	the sizes of the intersections between $c$ and all the other 
	columns. Suppose $c$ intersects $n_1$ columns in $\lambda_1$ 
	symbols and intersects $n_2$ columns in $\lambda_2$ symbols.
	Counting by columns, we then get $S = \lambda_1 n_1 + \lambda_2 n_2$.
	
	Counting by symbols present in column $c$, we get $S = k(k-1)$,
	since $c$ contains $k$ symbols and $A$ is equireplicate with 
	replication number $k$, that is, each of the $k$ symbols in $c$ 
	appears $k-1$ times outside of $c$.
	
	Equating the different counts, and using that $n_1+n_2=n-1$ and 
	$\lambda_1 + 1 = \lambda_2$, we get
	a linear equation $\lambda_1n_1 + \lambda_2(n-1-n_1) = k(k-1)$ 
	in the variable $n_1$, with solution $n_1 = \lambda_2(n-1)-k(k-1)$.
	It follows that $n_2  = - \lambda_1(n-1) + k(k-1)$, and since the 
	choice of $c$ was not used, these values are equal for all columns $c$.
\end{proof}

Theorem~\ref{TsuThm} together with Proposition~\ref{prop_nYR_col_intersect} 
gives the following:

\begin{theorem}\label{dual}
	If $D$ is the block design obtained from a $k \times n$ NYR, 
	then any pair of symbols is covered by either $\lambda_1$ or 
	$\lambda_2$ blocks in $D$.
\end{theorem}

\begin{proof}
	Let us first note that the left hand side of the identity in Theorem 
	\ref{TsuThm} is a multiple of the variance of the covering numbers. 
	The average covering number is $\lambda$, which is not an integer. 
	Hence the smallest possible variance would be achieved if all covering 
	numbers are one of $\lambda_1$ and $\lambda_2$. Since the variance is 
	a convex function this minimum is also unique.
	
	Using the values of $n_1$ and $n_2$ from 
	Proposition~\ref{prop_nYR_col_intersect} we can compute the right hand 
	side of the identity in Theorem \ref{TsuThm}. Using $\lambda_i$ and 
	$n_i$ for the covering numbers and their multiplicity in the left hand 
	side produces the same value.
	
	Hence the unique way to achieve the identity in Theorem \ref{TsuThm} 
	is to have all covering numbers equal to one of $\lambda_1$ and 
	$\lambda_2$, with the stated frequencies for those two numbers.
\end{proof}

Thus the block design coming from a NYR has both intersection numbers and 
covering numbers taking only two possible values, doing so in the way that 
minimises the variance of those numbers. The block designs appearing here 
are in fact members of a known class, introduced by John and Mitchell in 
1977 \cite{MR501627} called \emph{regular graph designs}. The name comes 
from a property of their concurrence matrices which is in fact the dual 
of our Proposition \ref{prop_nYR_col_intersect}. 
The class of regular graph designs, which includes non-symmetric designs, 
was later generalised to cases where equal replication is not possible, 
by Cheng and Wu in~\cite{MR626412}. As far as we know, Theorem \ref{dual} 
has however not been noticed in the literature on regular graph designs. 
Analogous to how Smith 
and Hartley~\cite{HS48} connect SBIBDs to Youden rectangles  we can obtain 
near Youden rectangles from regular graph designs by ordering the blocks 
and their elements. In fact, as observed e.g. by Bailey in Chapter 11.10 
of~\cite{Bailey08}, it is possible to order the elements of the blocks to 
become the columns of a row-column design for any equireplicate 
incomplete-block design with the same number of symbols as blocks.

Here we may also note that other types of designs where one allows the 
covering numbers, or intersection numbers to be non-constant have been 
studied. Bose and Nair~\cite{BoseNair} introduced and studied 
\emph{partially balanced incomplete block designs} (PBIBD). The 
particular case where there are just two different values for the number 
of repetitions of pairs in a PBIBD was studied for example by Bose and
Shimamoto in~\cite{BoseShimamoto}, and \emph{symmetric} PBIBDs have also been 
studied, e.g., by Lawless and Stanton in~\cite{SPBIBD}. Looking instead at sizes of intersections 
between blocks, the subclass of \emph{balanced incomplete block designs} 
(BIBDs) where the block intersections only have two different sizes has 
been studied under the name \emph{quasi-symmetric} designs, for example 
by Shrikande and Sane in~\cite{quasi}. Duals (exchanging the role of blocks and symbols) of 
PBIBDs have been studied under the name of \emph{linked block} designs (LB), 
with different relaxations, see, e.g., \cite{duals, Nair, RoyLaha}.

%----------------------------------------------------------------------
%----------------------------------------------------------------------
\section{Generating Data}
\label{sec_problem}
%----------------------------------------------------------------------
%----------------------------------------------------------------------

In this section, we describe our computational work in general terms.

%----------------------------------------------------------------------
\subsection{Guiding Questions}
%----------------------------------------------------------------------

Our approach is complete enumeration by computer for as large parameter 
values as possible, and unless otherwise stated, we save all generated data. 
In particular, we not only record the \emph{number} of Youden rectangles 
found, but we save the objects themselves.

With some exceptions due to size restrictions, the data generated 
is available for download at~\cite{Web3} and~\cite{Web4}.
Further details about the organization of the data are given there.

The following questions serve as guides for what data to generate.

\begin{description}
\item[(Q1)] How many isotopism classes of $k\times n$ Youden 
	rectangles are there?
\item[(Q2)] What is the order of the autotopism group of each 
	$k\times n$ Youden rectangle?
\item[(Q3)] If some condition is relaxed, how many objects satisfying the
	relaxed conditions are there?
\end{description}

%----------------------------------------------------------------------
\subsection{Feasible parameter combinations}
%----------------------------------------------------------------------

A necessary (but not sufficient) condition for the existence of a 
Youden rectangle is that $\lambda=\frac{k(k-1)}{n-1}$ is an integer. 
We exclude $k=1$, $k=n-1$ and $k=n$, as being trivial, since all Latin 
rectangles for those values are Youden rectangles. We call non-trivial 
parameter values satisfying the divisibility condition \emph{feasible}.
The smallest feasible parameter combinations for nontrivial Youden
rectangles are given in Table~\ref{table_param}. Note that if $(k,n)$
are feasible parameters for a Youden rectangle, then so are $(n-k,n)$.

\begin{table}[h]
\begin{center}
\begin{tabular}{|l|c|c|c|c|c|c|c|c|c|c|c|c|c|c|c|c|}
	\hline
	$n \backslash k$ 
		 &  3 & 4 & 5 & 6 & 7 & 8 & 9 & 10 & 11 & 12 & 16 \\ 
	\hline
	$7$  &  E & E &   &   &   &   &   &    &    &    &    \\ 
	\hline
	$11$ &    &   & E & E &   &   &   &    &    &    &    \\ 
	\hline
	$13$ &    & E &   &   &   &   & X &    &    &    &    \\ 
	\hline
	$15$ &    &   &   &   & X & X &   &    &    &    &    \\ 
	\hline
	$16$ &    &   &   & E? &   &   &   & X  &    &    &    \\ 
	\hline
	$19$ &    &   &   &   &   &   & X & X  &    &    &    \\ 
	\hline
	$21$ &    &   & E &   &   &   &   &    &    &    & X  \\ 
	\hline
	$23$ &    &   &   &   &   &   &   &    & X  & X  &    \\ 
	\hline 
\end{tabular} 
\end{center}
\caption{All feasible parameter combinations for Youden rectangles with $7 \leq n \leq 23$.
	An E indicates full enumeration in the present paper, and an
	X indicates feasible parameters but no complete enumeration.}
\label{table_param}
\end{table}

We have attempted to generate the Youden rectangles for all parameter 
combinations in Table~\ref{table_param}, but in the remaining cases 
the number of partial objects was too large and the computation had 
to be stopped due to lack of storage space. The fact that we could 
handle one case for $n=21$ illustrates the fact that growing $n$ is 
not the only challenge for complete enumeration, but rather an interplay 
between $n$ and $k$.

For parameter sets where there do not exist Youden rectangles, we have 
enumerated $k \times n$ near Youden rectangles, where the 
intersection sizes between symbol sets in columns (as noted above) are 
either $\lambda_1 = \left\lfloor \frac{k(k-1)}{n-1} \right\rfloor$ or 
$\lambda_2 = \left\lceil \frac{k(k-1)}{n-1} \right\rceil$. 

For near Youden rectangles there are no simple divisibility conditions 
which have to be satisfied, like the ones for SBIBDs, and as we shall 
see we find numerous examples for all small parameters. However, a theorem 
of Brown~\cite{Brown88} implies that for $n=17, k=6$, near Youden 
rectangles do not exist. So, while near Youden rectangles are much less 
restricted than Youden rectangles, the existence question is still non-trivial.

%----------------------------------------------------------------------
\subsection{Implementation and Execution}
%----------------------------------------------------------------------

We generated all non-isotopic rectangles by consecutively adding all possible columns, 
while observing that none of the conditions were violated. At suitable points,
we reduced our list of partial objects by isotopism. Also, at selected stages, 
the list of partial objects was culled by running checks on whether they
were at all extendible to a full Youden rectangle. We note here that using
the definition in terms of constant sized column intersections, rather than
the definition in terms of each symbol pair appearing a constant number of
times, makes it possible to reduce the list of partial objects much more
effectively. We also observe that for partial objects, it is not possible 
(at least not straightforwardly) to reduce the list of partial objects with 
respect to species (main classes). At this stage, therefore, it is natural
to employ the equivalence notion of isotopism.

The algorithms used were implemented in C++ and run in a parallelized version 
on the Kebnekaise supercomputer at High Performance Computing Centre 
North (HPC2N). 

The algorithm is divided into two parts. The first extends a given partial 
rectangle with $k$ rows and $t$ columns by one column, such that the new
rectangle satisfies both the Latin condition and the balance condition.
More specifically, we first add a column with $k$ different symbols. We then 
check that in the extended $k \times (t+1)$ rectangle, no symbols appear 
more than once in any row. We also check that the number of shared symbols 
between the added column and the $t$ first columns is $\lambda$. In the
case of generating near Youden rectangles, we instead
check that all intersection sizes with the new column fall into
one of the two allowed values.
By checking all possible added columns, we find all extensions of the 
given $k \times t$ rectangle.

The second part of the algorithm checks whether a received 
$k \times (t+1)$ rectangle could be chosen as a normalized 
representative of an isotopism class. 

When a full Youden rectangle has been received, we check the order of the 
autotopism group. The group of possible autotopism actions on a $k \times n$ 
Youden rectangle is $S_k\times S_n \times S_n$, so potentially, the number 
of actions we need to check is $k!\cdot n!\cdot n!$.

Since we consider normalized rectangles this number can be reduced to 
$n \cdot k!\cdot (n-k)!$, since once we have chosen the first column 
($n$ options) and row permutation $\pi_r$ ($k!$ options) we fix $k$ symbols 
in the symbol permutation $\pi_s$ (so $(n-k)!$ options remain).

The running time grows quickly as the rectangle parameters grow. 
We completely enumerated Youden rectangles of sizes $3\times 7$, $4\times 7$, 
$5\times 11$ and $4\times 13$ in a few minutes on a standard desktop computer.
On the other hand, computation on sizes $6\times 11$ and $5\times 21$ required 
high performance computers and significantly more time. Using 
a parallelized version of the algorithms, enumerating $6\times 11$ Youden 
rectangles took about 6000 core hours, which is a bit less than 1 year. The 
$5 \times 21$ case required several hundred core years. 

Our methods and code can be applied to larger parameter values as 
well,  but the number of partial rectangles, which is far larger than those for complete 
rectangles, become unmanageable.   The running time per  partial object has not been 
the bottleneck for our program, so there has been no reason to employ more sophisticated 
generation methods or equivalence checks. Instead the number of partial objects for large 
parameters became so large that disc space became the limiting factor.

%----------------------------------------------------------------------
%----------------------------------------------------------------------
\section{Basic Computational Results}
\label{sec_results}
%----------------------------------------------------------------------
%----------------------------------------------------------------------

We now turn to the results and analysis of our computational work. 

%----------------------------------------------------------------------
\subsection{The Number of Youden Rectangles}
\label{ssec_numberY}
%----------------------------------------------------------------------

Our first result is an enumeration of Youden rectangles. In 
Tables~\ref{ytab1} to~\ref{ytab4},
we present data on the number of non-isotopic Youden rectangles, 
sorted by the order of the autotopism groups.

It is relevant to compare these numbers with the number of Latin 
rectangles. When no reduction at all is applied, there are
\num{782 137 036 800} Latin rectangles of size $4 \times 7$,
and only $512$ Youden rectangles of the same size (note that this 
number is not given in any of the tables in the present paper). In 
\cite{MW}, the numbers of \emph{reduced} $n \times k$ Latin rectangles 
are given for $k \leq n$, $1 \leq n \leq 11$, that is, the number of 
Latin rectangles whose first row is the identity permutation and the 
first column is $0, 1, \ldots, k-1$, and there are 
\num{1 293 216} reduced Latin rectangles of size $4 \times 7$. 
Finally, there are \num{1398} $4 \times 7$ non-isotopic 
Latin rectangles~\cite{McKay98}, to be compared with
only 6 non-isotopic Youden rectangles of the same size.
As we can see, the proportion of Latin rectangles 
that additionally satisfy the balance condition is small.

We note again that the $3 \times 7$ and $4 \times 7$ Youden rectangles
were completely classified by Preece~\cite{Preece66}, and that
our enumerative results are in accordance with his classification.

\begin{table}

\begin{center}
\begin{tabular}{|c|r||c|c|}
	\hline
	\multicolumn{2}{|c||}{$(n,k,\lambda)$}	& (7,3,1) & (7,4,2) \\ \hline 
	\multicolumn{2}{|c||}{\#YR}	& 1 & 6 \\ \hline\hline
	& 1 & 0 & 2 \\ \cline{2-4}
	$|\mbox{Aut}|$ & 3 & 0 & 3 \\ \cline{2-4}
	& 21 & 1 & 1 \\ 
	\hline 
\end{tabular} 
\end{center}
\caption{The number of Youden rectangles with \mbox{$n=7$} sorted
	by autotopism group order.}
\label{ytab1}
\end{table}

\begin{table}

\begin{center}
\begin{tabular}{|c|r||r|r|}
	\hline
	\multicolumn{2}{|c||}{$(n,k,\lambda)$}	& (11,5,2) & (11,6,3) \\ \hline 
	\multicolumn{2}{|c||}{\#YR}	& 79 416 & 995 467 440 \\ \hline\hline
	& 1 & 77 694 & 995 421 832 \\ \cline{2-4}
	& 2 & 1 423 & 40 831 \\ \cline{2-4}
	& 3 & 199 & 4 454 \\ \cline{2-4}
	& 4 & 45 & 124 \\ \cline{2-4}
	$|\mbox{Aut}|$ & 5 & 4 & 121 \\ \cline{2-4}
	& 6 & 38 & 62 \\ \cline{2-4}
	& 10 & 3 & 3 \\ \cline{2-4}
	& 12 & 7 & 10 \\ \cline{2-4}
	& 55 & 1 & 1 \\ \cline{2-4}
	& 60 & 2 & 2 \\ 
	\hline 
\end{tabular} 
\end{center}
\caption{The number of Youden rectangles with \mbox{$n=11$} sorted
	by autotopism group order.}
\label{ytab2}
\end{table}

\begin{table}

\begin{center}
\begin{tabular}{|c|r||c|}
	\hline
	\multicolumn{2}{|c||}{$(n,k,\lambda)$}	& (13,4,1) \\ \hline 
	\multicolumn{2}{|c||}{\#YR}	& 20 \\ \hline\hline
	& 1 & 12 \\ \cline{2-3}
	$|\mbox{Aut}|$ & 3 & \phantom{0}7 \\ \cline{2-3}
	& 39 & \phantom{0}1  \\ 
	\hline 
\end{tabular} 
\end{center}
\caption{The number of Youden rectangles with \mbox{$n=13$, $k=4$} sorted
	by autotopism group order.}
\label{ytab3}
\end{table}

\begin{table}

\begin{center}
\begin{tabular}{|c|r||r|}
	\hline
	\multicolumn{2}{|c||}{$(n,k,\lambda)$}	& (21,5,1) \\ \hline 
	\multicolumn{2}{|c||}{\#YR}	& 3 454 435 044 \\ \hline\hline
	& 1 & 3 454 384 100 \\ \cline{2-3}
	& 2 & 37 394 \\ \cline{2-3}
	& 3 & 13 349 \\ \cline{2-3}
	& 5 & 14 \\ \cline{2-3}
	& 6 & 109 \\ \cline{2-3}
	$|\mbox{Aut}|$ & 7 & 4 \\ \cline{2-3}
	& 9 & 55 \\ \cline{2-3}
	& 14 & 6 \\ \cline{2-3}
	& 18 & 7 \\ \cline{2-3}
	& 21 & 1 \\ \cline{2-3}
	& 42 & 3 \\ \cline{2-3}
	& 63 & 1 \\ \cline{2-3}
	& 126 & 1 \\ 
	\hline 
\end{tabular} 
\end{center}
\caption{The number of Youden rectangles with \mbox{$n=21$, $k=5$} sorted
	by autotopism group order.}
\label{ytab4}
\end{table}

The most common autotopism group order for Latin rectangles is 1 (see
\cite{MW}). From the tables, we see that clearly the most common autotopism 
group order for Youden rectangles is also $1$, but that there are also rare 
examples of rather symmetric Youden rectangles. One such example, a Youden 
rectangle of size $4\times 13$, 
whose autotopism group order is $39$ is presented in 
Figure~\ref{YA13}. The autotopism group acts transitively on the columns 
of this Youden rectangle, that is, for any pair $C_1$ and $C_2$ of columns,
there is an autotopism that takes $C_1$ to $C_2$. 

As is well known, taking a $(n,k,\lambda)$ difference set as first
column and producing the remaining columns by developing this first 
column, that is, consecutively adding $1$ to each entry, will produce 
a Youden rectangle. The autotopism group of the resulting Youden 
rectangle will then act transitively on the set of columns.
We conclude that for $n=7, 11, 13$, the very symmetric Youden rectangles 
we found, where the order of the autotopism group is divisible by the 
number of columns, correspond to those Youden rectangles generated from 
difference sets. The situation for $n=21$ seems to be a bit more 
involved, since we see autotopism groups of orders $21$, $42$ (in fact, 
three such), $63$ and even $126$. A complete analysis of these 
Youden rectangles is beyond the scope of this paper, and we leave this 
as an open question.

\begin{figure}
	\begin{center}
		\begin{tabular}{|rrrrrrrrrrrrr|}
		\hline 
		0 & 1 & 2 & 3 & 4 & 5 & 6 & 7 & 8 & 9 & 10 & 11 & 12  \\ 
		%\hline 
		1 & 4 & 5 & 6 & 2 & 9 & 10 & 11 & 0 & 7 & 8 & 12 & 3  \\ 
		%\hline 
		2 & 5 & 7 & 8 & 9 & 11 & 0 & 3 & 4 & 12 & 1 & 6 & 10  \\ 
		%\hline 
		3 & 6 & 8 & 9 & 10 & 0 & 7 & 4 & 12 & 1 & 11 & 2 & 5  \\ 
		\hline 
	\end{tabular} 
	\end{center}
\caption{The $4\times 13$ Youden rectangle $Y$ with $|\mbox{Aut}(Y)|= 39$.}
\label{YA13}
\end{figure} 

For larger parameters, that is, where there exist more than one 
corresponding SBIBD, it would also have been interesting to group 
Youden rectangles according to which SBIBD they give if the ordering
in the columns is ignored.

%----------------------------------------------------------------------
\subsection{Near Youden Rectangles}
\label{ssec_twolambda}
%----------------------------------------------------------------------

In Tables~\ref{ntab1} to~\ref{ntab5}, we list complete data for the 
number of isotopism classes 
of near Youden rectangles (NYR) from $n=5$ to $n=9$ for sets of 
parameters where there are no Youden rectangles, sorted by the order of the 
autotopism groups.  We also display the number of NYRs which are 
self-conjugate as Latin rectangles, i.e., if we interchange the roles of 
columns and symbols we get a NYR in the same isotopism class. 

We have excluded the cases $k=1$, $k=n-1$ and $k=n$, since 
as observed above, for these cases all Latin rectangles are Youden 
rectangles as well. We also excluded the case $k=7, n=9$, for which 
the number of partial rectangles was deemed too large for a 
straight-forward run of our program.

In Tables~\ref{ntab6} to~\ref{ntab9}, we list data for the number of 
isotopism classes of near Youden rectangles from $n=10$ to $n=13$ for 
as large $k$ as was feasible, with the same restrictions on parameter 
values as for $n=5,\ldots, 9$.

\begin{observation}\label{obs_PBYexist}
	There exist NYRs for all parameters with $n\leq 10$.
\end{observation}

This follows from our enumeration together with the observation that 
if a $k \times n$ NYR is completed to an $n \times n$ Latin square then 
the new $n-k$ rows also form an $(n-k) \times n$ NYR.

We note that for $k=2$ any NYR may be interpreted as a 2-regular 
graph. Such graphs can be easily enumerated by hand, and our data for this
case is verified by such a manual count.

\begin{table}
	\begin{center}
		\begin{tabular}{|c|r||c|c|}
			\hline
			\multicolumn{2}{|c||}{$(n,k,\lambda_1,\lambda_2)$}	& (5,2,0,1) & (5,3,1,2) \\ 
			\hline 
			\multicolumn{2}{|c||}{\# NYR} &  1	& 2 \\ 
			\multicolumn{2}{|c||}{\#self-conjugate} &  1	& 2 \\ 
			\hline\hline
			$|\mbox{Aut}|$ & 2 & 0 & 1 \\ \cline{2-4}
			& 10 & 1 & 1 \\ 
			\hline 
		\end{tabular} 
	\end{center}
	\caption{The number of near Youden rectangles with 
		\mbox{$n=5$} sorted	by autotopism group order.}
	\label{ntab1}
\end{table}

\begin{table}
	\begin{center}
		\begin{tabular}{|c|r||c|c|c|}
			\hline
			\multicolumn{2}{|c||}{$(n,k,\lambda_1,\lambda_2)$} & (6,2,0,1) 	& (6,3,1,2) & (6,4,2,3)\\ 
			\hline 
			\multicolumn{2}{|c||}{\#NYR} & 2 & 2 & 34 \\
			\multicolumn{2}{|c||}{\#self-conjugate} & 2 & 2 & 29 \\			 
			\hline\hline
			& 1 & 0&0 & \phantom{0}9 \\ \cline{2-5}
			& 2 & 0&0 & 11 \\ \cline{2-5}
			& 4 & 0&0 & \phantom{0}5 \\ \cline{2-5}
			$|\mbox{Aut}|$ & 6 & 0& 2 & \phantom{0}3 \\ \cline{2-5}
			& 12 & 1 & 0 & \phantom{0}4 \\ \cline{2-5}
			& 18 & 0 & 0 & \phantom{0}1 \\ \cline{2-5}
			& 36 & 1&0 & \phantom{0}1 \\ 
			\hline 
		\end{tabular} 
	\end{center}
	\caption{The number of near Youden rectangles with 
		\mbox{$n=6$} sorted	by autotopism group order.}
	\label{ntab2}
\end{table}

\begin{table}
	\begin{center}
		\begin{tabular}{|c|r||r|r|}
			\hline
			\multicolumn{2}{|c||}{$(n,k,\lambda_1,\lambda_2)$} & (7,2,0,1)	& (7,5,3,4) \\ 
			\hline 
			\multicolumn{2}{|c||}{\# NYR} &  2	& 5 205 \\ 
			\multicolumn{2}{|c||}{\# self-conjugate} &  2	& 2 778 \\ 			
			\hline\hline
			& \phantom{0}1 & 0 & 4 889 \\ \cline{2-4}
			$|\mbox{Aut}|$ 
			& \phantom{0}2 & 0 & 307 \\ \cline{2-4}
			& \phantom{0}4 & 0 & 8 \\ \cline{2-4} 
			& 14 & 1 & 1 \\ \cline{2-4} 
			& 24 & 1 & 0 \\ 
			\hline 
		\end{tabular} 
	\end{center}
	\caption{The number of near Youden rectangles with 
		\mbox{$n=7$} sorted	by autotopism group order.}
	\label{ntab3}
\end{table}

\begin{table}
	\begin{center}
		\begin{tabular}{|c|r||r|r|r|r|r|r|}
			\hline
			\multicolumn{2}{|c||}{$(n,k,\lambda_1,\lambda_2)$} & (8,2,0,1)	& (8,3,0,1) & 
			(8,4,1,2) & (8,5,2,3) & (8,6,4,5) \\ 
			\hline 
			\multicolumn{2}{|c||}{\# NYR} & 3 & 4 & 285 & 6 688 & 21 956 009 \\ 
			\multicolumn{2}{|c||}{\# self-conjugate} & 3 & 3 & 212 & 3 608 & 11 000 012 \\ 			
			\hline\hline
			& \phantom{0}1 & 0 & 0 & 173 & 6 204 & 21 905 896 \\ \cline{2-7}
			& \phantom{0}2 & 0 & 0 & 78 & 381 & 48 865 \\ \cline{2-7}
			& \phantom{0}3 & 0 & 0 & 0 & 37 & 0 \\ \cline{2-7}
			& \phantom{0}4 & 0 & 0 & 15 & 29 & 1 208 \\ \cline{2-7}
			& \phantom{0}5 & 0 & 0 & 0 & 0 & 24 \\ \cline{2-7}
			& \phantom{0}6 & 0 & 2 & 0 & 18 & 0 \\ \cline{2-7}
			$|\mbox{Aut}|$ 
			& \phantom{0}8 & 0 & 0 & 11 & 6 & 144 \\ \cline{2-7} 
			& 10 & 0 & 0 & 0 & 0 & 6 \\ \cline{2-7} 
			& 12 & 0 & 0 & 0 & 5 & 0 \\ \cline{2-7} 
			& 16 & 1 & 1 & 4 & 5 & 36 \\ \cline{2-7} 
			& 24 & 0 & 0 & 0 & 2 & 0 \\ \cline{2-7} 
			& 30 & 1 & 0 & 0 & 0 & 0 \\ \cline{2-7} 	
			& 32 & 0 & 0 & 4 & 0 & 6 \\ \cline{2-7} 
			& 48 & 0 & 1 & 0 & 1 & 0 \\ \cline{2-7} 
			& 64 & 1 & 0 & 0 & 0 & 4 \\
			\hline 
		\end{tabular} 
	\end{center}
	\caption{The number of near Youden rectangles with 
		\mbox{$n=8$} sorted	by autotopism group order.}
	\label{ntab4}
\end{table}

\begin{table}
	\begin{center}
		\begin{tabular}{|c|r||r|r|r|r|r|r|}
			\hline
			\multicolumn{2}{|c||}{$(n,k,\lambda_1,\lambda_2)$}& (9,2,0,1)	& (9,3,0,1) & 
			(9,4,1,2) & (9,5,2,3) & (9,6,3,4) \\ \hline 
			\multicolumn{2}{|c||}{\# NYR} & 4	& 11 & 5 342 & 2 757 904 & 731 801 066 \\ 
			\multicolumn{2}{|c||}{\# self-conjugate} & 4	& 11 & 2 955 & 1 388 084 & 98 054 401 \\ 			
			\hline\hline
			& \phantom{0}1 & 0 & 3 & 4 881 & 2 750 174 & 731 727 683 \\ \cline{2-7}
			& \phantom{0}2 & 0 & 1 & 355 & 7 148 & 69 733 \\ \cline{2-7}
			& \phantom{0}3 & 0 & 1 & 20 & 290 & 3 079 \\ \cline{2-7}
			& \phantom{0}4 & 0 & 0 & 54 & 177 & 312 \\ \cline{2-7}
			& \phantom{0}6 & 0 & 4 & 15 & 86 & 213 \\ \cline{2-7}
			& \phantom{0}8 & 0 & 0 & 3 & 7 & 0 \\ \cline{2-7}
			$|\mbox{Aut}|$ 
			& \phantom{0}9 & 0 & 1 & 3 & 6 & 16 \\ \cline{2-7} 
			& 12 & 0 & 0 & 8 & 6 & 18 \\ \cline{2-7} 
			& 18 & 1 & 0 & 2 & 8 & 5 \\ \cline{2-7} 
			& 36 & 1 & 0 & 0 & 1 & 4 \\ \cline{2-7} 
			& 40 & 1 & 0 & 0 & 0 & 0 \\ \cline{2-7} 			
			& 54 & 0 & 1 & 0 & 0 & 1 \\ \cline{2-7} 
			& 72 & 0 & 0 & 1 & 1 & 0 \\ \cline{2-7} 
			& 108 & 0& 0 & 0 & 0 & 2 \\\cline{2-7} 
			& 324 & 1 & 0 & 0 & 0 & 0 \\
			\hline 
		\end{tabular} 
	\end{center}
	\caption{The number of near Youden rectangles with 
		\mbox{$n=9$} sorted	by autotopism group order.}
	\label{ntab5}
\end{table}

\begin{table}
	\begin{center}
		\begin{tabular}{|c|r||r|r|r|r|r|}
			\hline
			\multicolumn{2}{|c||}{$(n,k,\lambda_1,\lambda_2)$} & (10,2,0,1)	& (10,3,0,1) & 	(10,4,1,2) & (10,5,2,3)  \\ \hline 
			\multicolumn{2}{|c||}{\# NYR} & 5	& 80 & 9 722 & 1 913 816 \\ 
			\multicolumn{2}{|c||}{\# self-conjugate} & 5	& 59 & 5 388 & 962 300 \\ 			
			\hline\hline
			& \phantom{0}1 & 0 & 48 & 9 288 & 1 907 844  \\ \cline{2-6}
			& \phantom{0}2 & 0 & 23 & 331 & 5 952  \\ \cline{2-6}
			& \phantom{0}3 & 0 & 4 & 72 & 0  \\ \cline{2-6}
			& \phantom{0}4 & 0 & 0 & 9 & 0  \\ \cline{2-6}
			& \phantom{0}5 & 0 & 0 & 2 & 4  \\ \cline{2-6}
			& \phantom{0}6 & 0 & 2 & 2 & 0  \\ \cline{2-6}
			
			$|\mbox{Aut}|$ 
			
			& \phantom{0}10 & 0 & 3 & 9 & 16  \\ \cline{2-6} 			
			& 12 & 0 & 0 & 9 & 0  \\ \cline{2-6} 						
			& 20 & 1 & 0 & 0 & 0  \\ \cline{2-6} 			 
			& 42 & 1 & 0 & 0 & 0  \\ \cline{2-6} 
			& 48 & 1 & 0 & 0 & 0  \\ \cline{2-6} 
			& 100 & 1 & 0 & 0 & 0  \\ \cline{2-6} 		
			& 144 & 1 & 0 & 0 & 0  \\ 	
			\hline 
		\end{tabular} 
	\end{center}
	\caption{The number of near Youden rectangles with 
		\mbox{$n=10$} sorted by autotopism group order.}
	\label{ntab6}
\end{table}

\begin{table}
	\begin{center}
		\begin{tabular}{|c|r||r|r|r|r|}
			\hline
			\multicolumn{2}{|c||}{$(n,k,\lambda_1,\lambda_2)$} & (11,2,0,1)	& (11,3,0,1) & 
			(11,4,1,2)   \\ \hline 
			\multicolumn{2}{|c||}{\# NYR} & 6 & 852 & 1 598 \\ 
			\multicolumn{2}{|c||}{\# self-conjugate} & 6 & 501 & 865 \\ 			
			\hline\hline
			& \phantom{0}1 & 0 & 759 & 1 597  \\ \cline{2-5}
			& \phantom{0}2 & 0 & 75 & 0   \\ \cline{2-5}
			& \phantom{0}3 & 0 & 12 & 0  \\ \cline{2-5}
			& \phantom{0}6 & 0 & 5 & 0   \\ \cline{2-5}
			$|\mbox{Aut}|$ 				
			& \phantom{0}11 & 0 & 1 & 1  \\ \cline{2-5} 										
			& 22 & 1 & 0 & 0   \\ \cline{2-5} 	
			& 48 & 1 & 0 & 0   \\ \cline{2-5} 									
			& 56 & 1 & 0 & 0   \\ \cline{2-5} 									
			& 60 & 1 & 0 & 0   \\ \cline{2-5} 		
			& 180 & 1 & 0 & 0   \\ \cline{2-5} 					 						
			& 192 & 1 & 0 & 0   \\ \cline{2-5} 					
			\hline 
		\end{tabular} 
	\end{center}
	\caption{The number of near Youden rectangles with 
		\mbox{$n=11$} sorted by autotopism group order.}
	\label{ntab7}
\end{table}

\begin{table}
	\begin{center}
		\begin{tabular}{|c|r||r|r|r|r|}
			\hline
			\multicolumn{2}{|c||}{$(n,k,\lambda_1,\lambda_2)$}& (12,2,0,1)	& (12,3,0,1) & 
			(12,4,1,2)   \\ \hline 
			\multicolumn{2}{|c||}{\# NYR} & 9	& 11 598 & 262 \\ 
			\multicolumn{2}{|c||}{\# self-conjugate} & 9	& 6 183 & 167 \\ 			
			\hline\hline
			& \phantom{0}1 & 0 & 11 174 & 182  \\ \cline{2-5}
			& \phantom{0}2 & 0 & 333 & 46   \\ \cline{2-5}
			& \phantom{0}3 & 0 & 35 & 16  \\ \cline{2-5}
			& \phantom{0}4 & 0 & 13 & 4   \\ \cline{2-5}
			& \phantom{0}6 & 0 & 27 & 10   \\ \cline{2-5}
			& \phantom{0}8 & 0 & 2 & 0   \\ \cline{2-5}
			$|\mbox{Aut}|$ 			
			& 12 & 0 & 5 & 4   \\ \cline{2-5} 			
			& 18 & 0 & 3 & 0   \\ \cline{2-5}			
			& 24 & 1 & 4 & 0   \\ \cline{2-5}  
			& 54 & 1 & 0 & 0   \\ \cline{2-5} 		
			& 64 & 1 & 0 & 0   \\ \cline{2-5} 	
			& 70 & 1 & 0 & 0   \\ \cline{2-5} 				
			& 72 & 0 & 2 & 0   \\ \cline{2-5} 				
			& 120 & 1 & 0 & 0   \\ \cline{2-5} 				 						
			& 144 & 1 & 0 & 0   \\ \cline{2-5} 				
			& 216 & 1 & 0 & 0   \\ \cline{2-5} 				
			& 768 & 1 & 0 & 0   \\ \cline{2-5} 				
			& 388 & 1 & 0 & 0   \\ 
			\hline 
		\end{tabular} 
	\end{center}
	\caption{The number of near Youden rectangles with 
		\mbox{$n=12$} sorted by autotopism group order.}
	\label{ntab8}
\end{table}

\begin{table}
	\begin{center}
		\begin{tabular}{|c|r||r|r|}
			\hline
			\multicolumn{2}{|c||}{$(n,k,\lambda_1,\lambda_2)$}& (13,2,0,1)	& (13,3,0,1) \\ 
			\hline 
			\multicolumn{2}{|c||}{\# NYR} & 10	& 169 262 \\ 
			\multicolumn{2}{|c||}{\# self-conjugate} & 10	& 86 362 \\ 			
			\hline\hline
			& 1 & 0 & 167 541 \\ \cline{2-4}
			& 2 & 0 & 1 626 \\ \cline{2-4} 							
			$|\mbox{Aut}|$ & 3 & 0 & 69 \\ \cline{2-4}
			& 6  & 0& 24 \\ \cline{2-4}			
			& 13  & 0& 1 \\ \cline{2-4} 
			& 26  & 1& 0 \\ \cline{2-4}  											
			& 39  & 0& 1 \\  \cline{2-4}  											
			& 60  & 1& 0 \\ \cline{2-4}							
			& 72  & 1& 0 \\ \cline{2-4} 							
			& 80  & 1& 0 \\ \cline{2-4} 							
			& 84  & 1& 0 \\ \cline{2-4} 							
			& 144  & 1& 0 \\ \cline{2-4} 							
			& 252  & 1& 0 \\ \cline{2-4} 							
			& 300  & 1& 0 \\ \cline{2-4} 							
			& 320  & 1& 0 \\ \cline{2-4} 							
			& 1296  & 1& 0 \\ 
			\hline 
		\end{tabular} 
	\end{center}
	\caption{The number of near Youden rectangles with 
		\mbox{$n=13$} sorted by autotopism group order.}
	\label{ntab9}
\end{table}

We see that for fixed $n$ and growing $k$, at least for $n=7$, $n=11$ and
$n=13$, the number of near Youden rectangles grows faster 
than the number of Youden rectangles. The same holds for fixed $k$ and growing $n$.
As with Youden rectangles, most of the small near Youden 
rectangles have trivial autotopism groups.

We also note that for small $n$ we always find self-conjugate near Youden rectangles, even though their number is typically  smaller than
the number of all near Youden rectangles. 
\begin{question}
	Assume that a near Youden rectangle exists for  given $n$ and $k$. Does there always exist a self-conjugate near Youden rectangle for the same parameter combination?
\end{question}

%----------------------------------------------------------------------
%----------------------------------------------------------------------
\section{Relations to Triple Arrays and Related Row-Column Designs}
\label{sec_TA}
%----------------------------------------------------------------------
%----------------------------------------------------------------------

In this section, we present data and give some new theoretical results 
on the connection between Youden rectangles and double, triple and sesqui 
arrays.

%----------------------------------------------------------------------
\subsection{Theoretical background}
\label{subsec_TAinfo}
%----------------------------------------------------------------------

A $(v, e, \lambda_{rr}, \lambda_{cc}, \lambda_{rc} : r \times c)$ 
\emph{triple array} is an $r \times c$ array on $v$ symbols
satisfying the following conditions: 

\begin{description}
	\item[(TA1)]
		No symbol is repeated in any row or column.
	\item[(TA2)]
		Each symbol occurs $e$ times (the array is \emph{equireplicate}).
	\item[(TA3)]
		Any two distinct rows contain $\lambda_{rr}$ common symbols.
	\item[(TA4)]
		Any two distinct columns contain $\lambda_{cc}$ common symbols.
	\item[(TA5)]
		Any row and column contain $\lambda_{rc}$ common symbols.
\end{description}

If we relax condition (TA5), which is sometimes called 
\emph{adjusted orthogonality}, the array is called a \emph{double array}, 
and if condition (TA5) is expressly forbidden to hold, but all other 
conditions hold, we have a \emph{proper} double array. If we relax condition 
(TA4), the 
array is called a \emph{sesqui array}, and an array satisfying every condition
except (TA4) we call a \emph{proper} sesqui array. Our use of the term
\emph{proper} in this context should not be confused with how it
is sometimes used to stress that the blocks of a block design all 
have the same size.
Triple arrays were introduced by Agrawal~\cite{agrawal}, though examples
were known previously, and a good general introduction to triple and double
arrays is given in~\cite{McSorley}. Sesqui array were introduced in~\cite{sesqui}.

In discussing these designs we will find a new class of Latin rectangles 
useful.

\begin{definition}
	A Latin rectangle with integer parameters  $(n,k,\lambda)$, with $\lambda = \frac{k(k-1)}{n-1}$ 
	calculated from $n$ and $k$ as for a Youden rectangle, where the column 
	intersections have sizes $\lambda-1$, $\lambda$ and $\lambda+1$ is called a 
	\emph{triple-intersection} Latin rectangle.
\end{definition}

Note that these objects are defined only for such $(n,k,\lambda)$ that allow 
Youden rectangles with these parameters, and that we require the intersection
sizes to actually take on all these three values.

In~\cite{RagNag} it was suggested that triple arrays could be constructed
by taking an \emph{arbitrary} Youden rectangle, removing one column and all 
symbols present in that column, and then exchanging the roles of columns and 
symbols. The argument employed used distinct representatives. However, 
in~\cite{JW03}, the method was observed to be flawed, as the distinct 
representatives argument did not work, and an explicit counterexample was
given. For ease of reference, we phrase the construction as follows.

\begin{construction}\label{RagNag}
	For a given Youden rectangle $Y$ and a column $C_0$, let $A$ be the array
	received from $Y$ by first removing column $C_0$ and all occurrences in 
	$Y$ of symbols present in $C_0$, and then exchanging the roles of columns 
	and	symbols.
	
	We say that a Youden rectangle $Y$ is \emph{compatible} with an array $A$ 
	if $Y$ gives $A$ via this construction for some suitable choice of column,
	and we say that $Y$ \emph{yields} $A$.
\end{construction}

Construction~\ref{RagNag} was
further investigated in~\cite{NO15}, yielding among other the following 
results, reformulated to suit the terminology employed in the present paper:

\begin{theorem}[Proposition 2 in \cite{NO15}]\label{thm_alwayscolumn}
	Using Construction~\ref{RagNag}, any Youden rectangle  
	always yields an array that satisfies conditions (TA1), (TA2) and (TA4),
	regardless of the choice of column.
\end{theorem}

In particular, when applied to a $(n,k,\lambda)$ Youden rectangle, 
Construction~\ref{RagNag} yields an equireplicate 
$r \times c = k \times (n-k)$ array on 
$v=n-1$ symbols, with replication number $e = k-\lambda$ and column 
intersection size $\lambda_{cc} = \lambda$. We see then that 
Construction~\ref{RagNag}
may never (by definition of a proper sesqui array) yield a proper sesqui 
array, but it is possible that we would get the \emph{transpose} of a proper 
$(n-k) \times k$ sesqui array.

\begin{theorem}[Theorem 3 in \cite{NO15}]\label{thm_projectiveDA1}
	Using Construction~\ref{RagNag}, any Youden rectangle with $\lambda = 1$ 
	always yields a proper double array for any choice of column.
\end{theorem}

\begin{theorem}[Theorem 7 in \cite{NO15}]\label{thm_compatibleYR2}
	For any triple array $T$ with $v = r+c-1$ and $\lambda_{cc} = 2$, there
	exists a Youden rectangle (with $k = r$, $n=v+1$, $\lambda = 2$) that 
	yields $T$ using Construction~\ref{RagNag}.
\end{theorem}

It was also conjectured in \cite{NO15} that Theorem~\ref{thm_compatibleYR2} 
would hold for triple arrays with $\lambda_{cc}$ larger than 2.

When applying Construction~\ref{RagNag} to near Youden 
rectangles or triple-intersection Latin rectangles, removing a 
column together with all the symbols present in that column will leave
a $k \times (n-1)$ equireplicate array with some empty cells. For a near
Youden rectangle, the empty cells are distributed 
so that the number of empty cells in a column is either $\lambda_1$ or 
$\lambda_2$. For a triple intersection Latin rectangle, the corresponding 
numbers of empty cells are $\lambda-1$, $\lambda$ or $\lambda+1$. 
If more than one value occurs for the number of empty cells in a column, 
the array will not be equireplicate after exchanging the roles of columns 
and symbols, since the number of appearances of a symbol in the resulting 
array will be the number of non-empty cells in the corresponding column. 

For near Youden rectangles Proposition \ref{prop_nYR_col_intersect} 
implies that the resulting array will never be equireplicate.
However, the following theorem follows rather easily from results 
in \cite{NO15}.

\begin{theorem}\label{thm_3lambda}
	For any $(v, e, \lambda_{rr}, 1, \lambda_{rc} : r \times c)$ triple 
	array $T$ with $v=r+c-1$, there is a compatible $r \times (v-c)$ 
	triple-intersection Latin rectangle $Y$ with column intersection 
	sizes $0$, $1$ and $2$.
\end{theorem}

The proof of this theorem uses the following result, where the 
RL-\emph{form} $R$ of a triple array $T$ mentioned in the cited 
source is the array that results from exchanging the roles of columns 
and symbols in $T$.

\begin{theorem}[Corollary 1 in \cite{NO15}]\label{thm_sumconstant}
	In the RL-form $R$ of a triple array $T$ with $v = r+c-1$, 
	for any two columns $C_1$ and $C_2$, the sum of 
	the number of common non-empty rows and the number of common 
	symbols of $C_1$ and $C_2$ is constant, namely $e$, the 
	replication number.
\end{theorem}

\begin{proof}[Proof of Theorem~\ref{thm_3lambda}]
Since the parameters of $T$ are not all independent of each 
other (in particular, when $v=r+c-1$, it holds that $\lambda_{cc} = r-e$, 
see \cite{McSorley}), we may also observe that when exchanging the roles 
of symbols and columns in a $T$, there will be $r-e = \lambda_{cc}$ empty 
cells in each column in $R$ (the number of rows in $T$ in which the 
corresponding symbol does not appear). Reasoning similarly, there will be 
$r-\lambda_{cc}$ empty cells in each row of $R$ (the number of columns 
in $T$ where the corresponding symbol does not appear).

For $\lambda_{cc} = 1$, Theorem~\ref{thm_sumconstant} then implies that 
in $R$, each pair of columns shares $0$ symbols (when their empty 
cells lie in the same row) or $1$ symbol (when their empty cells lie in 
different rows).

With this information, given a triple array $T$, we can construct a Youden 
rectangle $Y$ compatible with $T$ by first exchanging the roles of columns 
and symbols in $T$, yielding the array $R$, 
and then adding a new column $C_0$ with a set $S$ of $r$ new symbols, 
$s_1, s_2, \ldots s_r$ in this order. To fill the empty cells in row 
$i$ in $R$, we then use the $r-1$ symbols $S \setminus \{s_i\}$,
in any order. This is the right number of symbols, since there are $r-1$ 
empty cells in every row of $R$, and there will be no repeated symbol in 
any row or column.

The intersections between columns in $Y$ may now have three different sizes. 
As observed above, pairs of columns in $R$ shared either $0$ symbols or $1$ symbol,
and after adding symbols to form $Y$, these numbers may have gone up by
at most $1$, since only one new symbol was added in each column. 
\end{proof}

An example of the construction in the above proof is given in 
Figure~\ref{fig_nearYconstr}.
Since Theorem~\ref{thm_3lambda} shows that the same transformation that
we applied to Youden rectangles could yield interesting row-column designs 
when applied to a triple-intersection Latin rectangle, we have also 
included this in our computational studies.

\begin{figure}[ht]
\begin{center}

\subfigure[A $4 \times 9$ triple array $T$.]{

	\begin{tabular}{|rrrrrrrrr|}
		\hline
		0	&2	&1	&4	&5	&6	&8	&7	&10 \\
		11	&3	&8	&5	&6	&7	&9	&1	&2	\\
		5	&7	&4	&9	&3	&11	&0	&10	&8	\\
		1	&0	&3	&2	&10	&4	&6	&9	&11	\\
		\hline
	\end{tabular}	
}

\bigskip
\hfill	

\subfigure[The corresponding array $R$ with the roles of symbols and rows in $T$ interchanged]{

	\begin{tabular}{|cccccccccccc|}
		\hline
		0	&2	&1	&	&3	&4	&5	&7	&6	&	&8	&	\\
			&7	&8	&1	&	&3	&4	&5	&2	&6	&	&0	\\
		6	&	&	&4	&2	&0	&	&1	&8	&3	&7	&5	\\
		1	&0	&3	&2	&5	&	&6	&	&	&7	&4	&8	\\	
		\hline
	\end{tabular}

}

\bigskip
\hfill	
	
\subfigure[A triple-intersection Latin rectangle compatible with $T$.]{

	\begin{tabular}{|rrrrrrrrrrrrr|}
		\hline
		9	&0	&2	&1	&10	&3	&4	&5	&7	&6	&11	&8	&12 \\
		10	&9	&7	&8	&1	&11	&3	&4	&5	&2	&6	&12	&0	\\
		11	&6	&9	&10	&4	&2	&0	&12	&1	&8	&3	&7	&5	\\
		12	&1	&0	&3	&2	&5	&9	&6	&10	&11	&7	&4	&8	\\
		\hline
	\end{tabular}

}
\end{center}

	\caption{Example of the construction in the proof of 
		Theorem~\ref{thm_3lambda}.}
	\label{fig_nearYconstr}
\end{figure}

%----------------------------------------------------------------------
\subsection{Computational Results for Youden Rectangles}
\label{subsec_TAdata}
%----------------------------------------------------------------------

In this section, we report on how many Youden rectangles yielded
triple arrays, proper double arrays, or transposes of proper sesqui arrays,
for all parameters $(n,k,\lambda)$ for which we have complete data,
except for $(21,5,1)$ Youden rectangles, where the computing time 
required was too great. 

We ran checks even for properties guaranteed by 
Theorems~\ref{thm_alwayscolumn}, \ref{thm_projectiveDA1} and 
\ref{thm_compatibleYR2}. Computational results were compatible with those 
of these theorems, which can be taken as an independent indication of 
the correctness of the computations.

%----------------------------------------------------------------------
\subsubsection{Triple Arrays}
%----------------------------------------------------------------------

Among the possible parameters for Youden rectangles for which we have
complete data, there are just two sets of parameters where there is 
a chance of producing triple arrays, namely $(11,5,2)$ and $(11,6,3)$.
All Youden rectangles with $\lambda = 1$ are excluded by 
Theorem~\ref{thm_projectiveDA1}, and $(7,4,2)$ would give a $4 \times 3$ 
triple array, the existence of which was excluded in~\cite{McSorley}.

In Table~\ref{tatab} for triple arrays and Table~\ref{dotab} for 
proper double arrays we give the following information:
\begin{enumerate} 
	\item The number of Youden rectangles that give a triple or 
		double array via Construction~\ref{RagNag} for 
		\emph{at least one} of its columns. 
	\item The total number of columns for which the construction
		yields a triple or double array (that is, Youden rectangles 
		counted with `multiplicities'). 
	\item The number of non-isotopic triple or proper double arrays we 
		observe appearing as a result of this operation.
\end{enumerate}

\begin{table}[ht]
\begin{center}
\begin{tabular}{|c|c|c|c|}
	\hline 
	$(n,k,\lambda)$& \# compatible YR &\# compatible columns & \# TA  \\ 
	\hline 	
	(11,5,2) & \phantom{0}52 & \phantom{0}52 & 7 \\ 
	\hline 	
	(11,6,3) & 826 & 826 & 7 \\ 
	\hline 	
\end{tabular} 
\end{center}
\caption{The number of Youden rectangles giving triple arrays.}
\label{tatab}
\end{table}

The $5 \times 6$ triple arrays (and by taking transposes, also the
$6 \times 5$ triple arrays) were completely classified into $7$
isotopism classes in~\cite{TA5x6}.
As predicted by Theorem~\ref{thm_compatibleYR2}, all $7$ triple arrays 
appear in Table~\ref{tatab}. 

\begin{observation}
	Each of the Youden rectangles with $n=11$ that yields a triple array 
	does so using a unique column.
\end{observation}

The $7$ different triple arrays do not appear equally often. With classes 
numbered as in~\cite{TA5x6}, the triple arrays appear with the frequencies 
given in Table~\ref{tabclass}. The orders of the autotopism groups of the 
triple arrays (in the row labelled TA $|\mbox{Aut}|$) are taken 
from~\cite{TA5x6}. It seems that it is easier to produce those triple arrays
that have smaller autotopism groups.

\begin{table}[ht]
\begin{center}
\begin{tabular}{|c||c|c|c|c|c|c|c|}
	\hline 
	TA class & 1 & 2 & 3 & 4 & 5 & 6 & 7 \\ 
	\hline 
	TA $|\mbox{Aut}|$ & 60 & 12 & 12 & 6 & 4 & 3 & 3 \\
	\hline	\hline
	\# $5 \times 6$ YR & \phantom{0}3 & \phantom{0}5 & \phantom{0}5 & \phantom{00}8 
	& \phantom{0}11 & \phantom{0}10 & \phantom{0}10 \\ 
	\hline	
	\# $6 \times 5$ YR & 23  & 62  & 62  & 115  & 168  & 198  & 198  \\ 
	\hline 	
\end{tabular} 
\end{center}
\caption{The number of Youden rectangles giving each of the $7$ classes
	of $5 \times 6$ triple arrays.}
\label{tabclass}
\end{table}

We investigated the autotopism group orders of the Youden rectangles that
produced triple arrays, but we observed no obvious patterns.

%----------------------------------------------------------------------
\subsubsection{Proper Double Arrays}
%----------------------------------------------------------------------

We also checked which Youden rectangles produced proper double 
arrays, and the results are given in Table~\ref{dotab}. 
As predicted by Theorem~\ref{thm_projectiveDA1}, we see that all Youden 
rectangles with $\lambda = 1$ produced proper double arrays, for each 
column. For other values of $\lambda$, there is some indication that
the proportion of compatible Youden rectangles decreases with growing
$\lambda$, and that the most common case is that even in a compatible
Youden rectangle, only one column is compatible.

\begin{table}[ht]
\begin{center}
\begin{tabular}{|c|c|c|c|}
	\hline 
	$(n,k,\lambda)$ & \# compatible YR & \# compatible columns  & \# DA    \\ 
	\hline 	
	(7,3,1) & \phantom{00 000 00}1 & \phantom{00 000 00}7 & \phantom{00 00} 1 \\ 
	\hline 	
	(7,4,2) & \phantom{00 000 00}6 & \phantom{00 000 0}18 & \phantom{00 00}2 \\ 
	\hline 	
	(11,5,2) & \phantom{00 0}44 012 & \phantom{00 0}64 949 & 17 642 \\ 
	\hline 	
	(11,6,3) & 31 782 790 & 32 335 774 & 24 663 \\ 
	\hline 	
	(13,4,1) & \phantom{00 000 0}20 & \phantom{00 000 }260 & \phantom{00 }192 \\ 
	\hline 	
\end{tabular} 
\end{center}
\caption{The number of Youden rectangles giving proper double arrays.}
\label{dotab}
\end{table}

We note also that for parameter pairs $(n,k, \lambda_1)$, 
$(k,n-k,\lambda_2)$, the double arrays produced by the first have dimensions 
$k \times (n-k)$ and taking transposes yields an $(n-k) \times k$ double array,
and vice versa. Despite this, we see different numbers of double arrays
appearing through the construction both for the pair $(7,3,1), (7,4,2)$
and the pair $(11,5,2),(11,6,3)$. This would seem to indicate that there
are double arrays that cannot be constructed using Construction~\ref{RagNag}.

We note that on the basis of these data, we can answer in the negative a 
question posed in~\cite{NO15}, namely whether every Youden rectangle gives 
a double array using Construction~\ref{RagNag} for \emph{some} column. We 
phrase this as an observation. For examples for $v=11$, see 
Figure~\ref{fig_notcomp}.

\begin{observation}\label{obs_notcomp}
	There are Youden rectangles that cannot be used to produce double arrays 
	by removing a column and all the symbols in that column, and then 
	interchanging the roles of symbols and columns.
\end{observation}

\begin{figure}[ht]
\begin{center}
\subfigure[An $11 \times 5$ Youden rectangle which does not	give a double array for any column]{
	\begin{tabular}{|ccccccccccc|}
		\hline
		0 & 1 & 2 & 3 & 4 & 5 & 6 & 7 & 8 & 9 & 10\\
		1 & 0 & 5 & 6 & 7 & 10 & 4 & 9 & 3 & 8 & 2\\
		2 & 5 & 0 & 9 & 8 & 3 & 10 & 4 & 6 & 1 & 7\\
		3 & 6 & 8 & 10 & 0 & 1 & 2 & 5 & 7 & 4 & 9\\
		4 & 7 & 9 & 0 & 10 & 8 & 5 & 3 & 2 & 6 & 1\\ 
		\hline
	\end{tabular}
}
\bigskip
\hfill	
	
\subfigure[An $11 \times 6$ Youden rectangle which does not 	give a double array for any column]{
	\begin{tabular}{|ccccccccccc|}
		\hline
		0 & 1 & 2 & 3 & 4 & 5 & 6 & 7 & 8 & 9 & 10\\
		1 & 0 & 9 & 4 & 7 & 8 & 10 & 5 & 3 & 6 & 2\\
		2 & 3 & 4 & 7 & 5 & 6 & 1 & 8 & 10 & 0 & 9\\
		3 & 6 & 7 & 1 & 8 & 4 & 5 & 9 & 2 & 10 & 0\\
		4 & 7 & 0 & 9 & 10 & 1 & 2 & 3 & 6 & 5 & 8\\
		5 & 8 & 6 & 10 & 0 & 9 & 7 & 2 & 4 & 3 & 1\\
		\hline
	\end{tabular}
}
\end{center}
	\caption{Examples for Observation~\ref{obs_notcomp}.}
	\label{fig_notcomp}
\end{figure}

%----------------------------------------------------------------------
\subsubsection{Transposes of Proper Sesqui Arrays}
%----------------------------------------------------------------------

Using Construction~\ref{RagNag}, we checked for transposes of proper sesqui 
arrays, and the results are presented in Table~\ref{sqtab}.

\begin{table}[ht]
\begin{center}
\begin{tabular}{|c|c|c|c|}
	\hline 
	$(n,k,\lambda)$ & \# compatible YR & \# compatible columns  & \# SA$^T$  \\ 
	\hline 	
	(7,3,1)  & \phantom{00 000 00}0 & \phantom{00 000 00}0 & \phantom{00 00}0 \\ 
	\hline 	
	(7,4,2)  & \phantom{00 000 00}1 & \phantom{00 000 00}3 & \phantom{00 00}1 \\ 
	\hline 	
	(11,5,2) & \phantom{00 000 00}0 & \phantom{00 000 00}0 & \phantom{00 00}0 \\
	\hline 	
	(11,6,3) & \phantom{00 00}8 234 & \phantom{00 00}8 234 & \phantom{00 0}34 \\ 
	\hline 	
	(13,4,1) & \phantom{00 000 00}0 & \phantom{00 000 00}0 & \phantom{00 00}0 \\
	\hline 	
\end{tabular} 
\end{center}
\caption{The number of Youden rectangles giving transposes of 
	proper sesqui arrays.}
\label{sqtab}
\end{table}

We observe that transposes of sesqui arrays are relatively rare, and that
the compatible $(11,6,3)$ Youden rectangles are only compatible for one
single column each. The one compatible $(7,4,2)$ Youden rectangle is given in 
Figure~\ref{7x4SAt}, together with the resulting transposed sesqui array.

\begin{figure}
\begin{center}	
\subfigure[The Youden rectangle.]{
	\begin{tabular}{|ccccccc|}
	\hline
	0 & 1 & 2 & 3 & 4 & 5 & 6 \\
	1 & 2 & 3 & 4 & 6 & 0 & 5 \\
	2 & 4 & 5 & 6 & 0 & 3 & 1 \\
	3 & 5 & 6 & 1 & 2 & 4 & 0 \\
	\hline \hline
	  & S & S &   &   & D & S \\
	\hline
	\end{tabular}
	\label{YA7SAta} 
}
\subfigure[The transposed sesqui array.]{
    \begin{tabular}{|ccc|}
		\hline
		0 & 1 & 4 \\
		1 & 4 & 2 \\
		2 & 3 & 5 \\
		3 & 5 & 0 \\
		\hline
	\end{tabular} 
	 \label{YA7SAtb}
}
\end{center}
	\caption{The unique $4\times 7$ Youden rectangle compatible with the
		transpose of a sesqui array, with compatible columns marked 
		by S, and a column compatible with a double array marked
		by D.}
	\label{7x4SAt}
\end{figure}

%----------------------------------------------------------------------
\subsubsection{Compatibility with Several Designs}
%----------------------------------------------------------------------

In our data, we found some specimens of Youden rectangles exhibiting 
very good compatibility properties.
To begin with, in Figure~\ref{7x4SAt}, we give a $(7,4,2)$ Youden 
rectangle which is compatible both with transposes of sesqui arrays,
and with a proper double array.

Further, some of the Youden rectangles that gave triple arrays
of dimensions $5 \times 6$ and $6 \times 5$
also gave proper double arrays for some other columns. Examples 
with maximum number of columns compatible with double arrays are
given in Figures~\ref{TAandDA5} and \ref{TAandDA6}.

\begin{figure}
	\begin{center}
		\begin{tabular}{|ccccccccccc|}
			\hline
			0 & 1 & 2 & 3 & 4 & 5 & 6 & 7 & 8 & 9 & 10\\
			1 & 0 & 3 & 7 & 6 & 8 & 9 & 10 & 5 & 2 & 4\\
			2 & 5 & 7 & 9 & 0 & 3 & 1 & 8 & 10 & 4 & 6\\
			3 & 6 & 8 & 10 & 9 & 1 & 2 & 4 & 0 & 7 & 5\\
			4 & 7 & 6 & 0 & 8 & 9 & 10 & 1 & 2 & 5 & 3\\
			\hline \hline
			  & T & D &   & D &   & D  &   &   &   & D\\
			\hline
		\end{tabular}       
	\end{center}
	\caption{Example of a $5\times 11$ Youden rectangle with maximum 
		compatibility with respect to triple and proper double arrays. The 
		column marked with T is compatible with a triple array, and
		the four columns marked with D are compatible with proper double
		arrays.}
	\label{TAandDA5}
\end{figure} 

\begin{figure}
	\begin{center}
		\begin{tabular}{|ccccccccccc|}
			\hline
				0 & 1 & 2 & 3 & 4 & 5 & 6 & 7 & 8 & 9 & 10\\ 
				1 & 0 & 4 & 10 & 5 & 7 & 8 & 2 & 3 & 6 & 9\\ 
				2 & 3 & 6 & 9 & 7 & 8 & 0 & 10 & 1 & 5 & 4\\ 
				3 & 6 & 9 & 5 & 10 & 0 & 4 & 1 & 2 & 8 & 7\\ 
				4 & 7 & 3 & 6 & 8 & 9 & 2 & 5 & 10 & 1 & 0\\ 
				5 & 8 & 7 & 0 & 3 & 2 & 10 & 6 & 9 & 4 & 1\\
			\hline \hline
			  &  &  & D  &  &   &   &   &   & T  & \\
			\hline
		\end{tabular}       
	\end{center}
	\caption{Example of a $6\times 11$ Youden rectangle with maximum 
		compatibility with respect to triple and proper double arrays. The 
		column marked with T is compatible with a triple array, and
		the column marked with D is compatible with a proper double
		array.}
	\label{TAandDA6}
\end{figure}

Even for Youden rectangles with $\lambda \neq 1$, we found Youden 
rectangles that for each column are compatible with some proper double array.

In Figure~\ref{YA7DA}, we give the unique $4\times 7$ Youden rectangle 
where each column is compatible with a double array. For any column, 
the resulting double array is isotopic to the one given  
in Figure~\ref{YA7DAb}.
The Youden rectangle in Figure~\ref{YA7DAa} has the largest autotopism 
group order, i.e., $21$, and the autotopism group acts transitively 
on the columns. As observed above, this Youden rectangle can therefore
be produced from a difference set. The double array has an autotopism 
group of order $3$, which acts transitively on the columns.

\begin{figure}
\begin{center}
\subfigure[The Youden rectangle.]{
	\begin{tabular}{|ccccccc|}
	\hline
	0 & 1 & 2 & 3 & 4 & 5 & 6\\
	1 & 2 & 4 & 5 & 3 & 6 & 0\\
	2 & 4 & 3 & 6 & 5 & 0 & 1\\
	3 & 5 & 6 & 1 & 0 & 2 & 4\\
	\hline
	\end{tabular}
	\label{YA7DAa}
}
%  \hfill
\subfigure[The double array.]{
    \begin{tabular}{|ccc|}
		\hline
		0 & 1 & 3 \\
		1 & 2 & 5\\
		2 & 4 & 0\\
		3 & 5 & 4\\
		\hline
	\end{tabular} 
		 \label{YA7DAb}
	}
\end{center}	

	\caption{The unique $4\times 7$ Youden rectangle where each 
		column is compatible with a double array.}
	\label{YA7DA}
\end{figure}

For $n=11$, the situation is a bit more complicated.
In Figure~\ref{YA11}, we give the two $5 \times 11$ examples we found, 
and in Figure~\ref{YA116}, we give the unique $6 \times 11$ example.

\begin{figure}
\begin{center}
	\subfigure[]{
		\centering
		\begin{tabular}{|ccccccccccc|}
			\hline 
			0 & 1 & 2 & 3 & 4 & 5 & 6 & 7 & 8 & 9 & 10\\ 
			1 & 2 & 5 & 6 & 7 & 3 & 8 & 9 & 4 & 10 & 0\\ 
			2 & 5 & 3 & 8 & 9 & 6 & 4 & 10 & 7 & 0 & 1\\
			3 & 6 & 8 & 7 & 0 & 4 & 9 & 1 & 10 & 2 & 5\\
			4 & 7 & 9 & 0 & 5 & 10 & 1 & 3 & 2 & 6 & 8\\ 
			\hline 
		\end{tabular} 
	\label{YA11a}
	}
		
	\hfill
	\bigskip
		
	\subfigure[]{
		\centering
		\begin{tabular}{|ccccccccccc|}
			\hline 
			0 & 1 & 2 & 3 & 4 & 5 & 6 & 7 & 8 & 9 & 10\\ 
			1 & 0 & 5 & 6 & 7 & 3 & 4 & 2 & 9 & 10 & 8\\ 
			2 & 5 & 0 & 8 & 9 & 4 & 10 & 6 & 1 & 3 & 7\\ 
			3 & 6 & 8 & 0 & 10 & 7 & 2 & 9 & 4 & 5 & 1\\ 
			4 & 7 & 9 & 10 & 0 & 8 & 5 & 3 & 6 & 1 & 2\\ 
			\hline 
		\end{tabular} 
	\label{YA11b}
	}
\end{center}
	\caption{The only two $5\times 11$ Youden rectangles where each 
		column is compatible with a proper double array.}
	\label{YA11}
\end{figure}  

\begin{figure}
	\begin{center}
		\begin{tabular}{|ccccccccccc|}
			\hline
			0 & 1 & 2 & 3 & 4 & 5 & 6 & 7 & 8 & 9 & 10\\
			1 & 2 & 6 & 4 & 7 & 8 & 3 & 5 & 9 & 10 & 0\\
			2 & 6 & 3 & 7 & 5 & 9 & 4 & 8 & 10 & 0 & 1\\
			3 & 4 & 7 & 8 & 9 & 0 & 5 & 10 & 1 & 2 & 6\\
			4 & 7 & 5 & 9 & 10 & 1 & 8 & 0 & 2 & 6 & 3\\
			5 & 8 & 9 & 0 & 1 & 6 & 10 & 2 & 3 & 4 & 7\\
			\hline
		\end{tabular}
	\end{center}
	\caption{The unique $6\times 11$ Youden rectangle where each column 
		is compatible with a proper double array.}
	\label{YA116}
\end{figure} 

The Youden rectangle in Figure~\ref{YA11a} has an autotopism group of 
order $55$, which acts transitively on the columns, and so comes from a 
difference set. All columns yield a double array isotopic to the one in 
Figure~\ref{DA11x5_1}. The autotopism group order of this double array is 
$5$, and it acts transitively on 5 of the columns, but keeps column 5 
fixed. 

\begin{figure}
\begin{center}
	\begin{tabular}{|cccccc|}
	\hline 
	0 & 1 & 2 & 3 & 4 & 6 \\ 
	1 & 2 & 5 & 6 & 7 & 8 \\ 
	2 & 5 & 3 & 8 & 9 & 4 \\ 
	3 & 6 & 8 & 7 & 0 & 9 \\ 
	4 & 7 & 9 & 0 & 5 & 1 \\
	\hline 
\end{tabular} 
\end{center}
\caption{The double array produced from the $5 \times 11$ Youden rectangle
	with autotopism group order $55$ given in Figure~\ref{YA11a}.}
\label{DA11x5_1}
\end{figure}

The Youden rectangle in Figure~\ref{YA11b} has an autotopism group of
order $60$, which acts transitively on two groups of columns, with $5$ 
and $6$ columns, respectively. 
All columns in the group with five columns yield the double array in 
Figure~\ref{DA11x5_2a}, and all columns in the group with six columns 
yield the double array in Figure~\ref{DA11x5_2b}.
The autotopism group order of these double arrays are $12$ and $10$, 
respectively, and the group action for the first one is transitive on the 
columns, while the autotopism group for the second one acts transitively
on all columns except the second column, which is fixed.

\begin{figure}
\begin{center}
\subfigure{
\begin{tabular}{|cccccc|}
\hline 
	0 & 1 & 2 & 5 & 6 & 9\\ 
	1 & 5 & 4 & 0 & 3 & 7\\ 
	2 & 3 & 6 & 4 & 0 & 8\\ 
	3 & 6 & 8 & 7 & 9 & 1\\ 
	4 & 7 & 5 & 9 & 8 & 2\\
	\hline 
\end{tabular} 
\label{DA11x5_2a}
}
\subfigure{
\begin{tabular}{|cccccc|}
\hline 
0 & 1 & 2 & 3 & 7 & 8\\ 
1 & 0 & 3 & 4 & 9 & 5\\ 
2 & 5 & 6 & 9 & 8 & 0\\ 
3 & 6 & 5 & 7 & 2 & 4\\ 
4 & 7 & 8 & 6 & 1 & 9\\
\hline
\end{tabular}
\label{DA11x5_2b}
}
\end{center}
\caption{The double arrays produced from the $5 \times 11$ Youden rectangle
	with autotopism group order $60$ given in Figure~\ref{YA11b}.}
\label{DA11x5_2}
\end{figure}

Finally, the Youden rectangle in Figure~\ref{YA116} has an autotopism 
group of order $55$, which acts transitively on the columns, and so comes 
from a difference set. All columns yield the same double array, given in 
Figure~\ref{DA11x6}, which has an autotopism group of order  5,
which acts transitively on the columns.

\begin{figure}
\begin{center}
	\begin{tabular}{|ccccc|}
	\hline 
0 & 1 & 2 & 3 & 5\\ 
1 & 2 & 6 & 4 & 8\\ 
2 & 6 & 3 & 7 & 9\\ 
3 & 4 & 7 & 8 & 0\\ 
4 & 7 & 5 & 9 & 1\\ 
5 & 8 & 9 & 0 & 6\\
	\hline 
\end{tabular}

\end{center}
\caption{The double array produced from the $6 \times 11$ Youden rectangle
	with autotopism group order $55$ given in Figure~\ref{YA116}.}
\label{DA11x6}
\end{figure}

It is interesting to note that the Youden rectangles in 
Figures~\ref{YA7DA}--\ref{YA116} that produce a single double array
(up to isotopism) for all columns have autotopism groups that act 
transitively on the columns. For an investigation of this topic,
we refer the interested reader to~\cite{TAdiff}.

%----------------------------------------------------------------------
\subsection{Computational Results for triple-intersection Latin Rectangles}
%----------------------------------------------------------------------

As we noted earlier, triple-intersection Latin rectangles both provide 
the missing 
source for the $\lambda=1$ triple arrays and could potentially lead to 
additional row-column designs. In order to investigate this connection we 
have also generated all triple-intersection Latin rectangles with $n=7$, 
but for larger $n$ we deemed full enumeration infeasible. The number of 
such rectangles 
is given in Table \ref{tlr}, sorted by the order of the autotopism groups. 

\begin{table}
	\begin{center}
		\begin{tabular}{|c|r||r|r|}
			\hline
			\multicolumn{2}{|c||}{$(n,k)$}	& (7,3) & (7,4) \\ 
			\hline 
			\multicolumn{2}{|c||}{\# TILR}	& 43 & 872\\ 
			\hline\hline
			& \phantom{0}1 & 18 & 756\\ \cline{2-4}
			$|\mbox{Aut}|$ 
			& \phantom{0}2 & 21 & 101\\ \cline{2-4}
			& \phantom{0}3 & 1 & 10\\ \cline{2-4}
			& \phantom{0}4 & 0 & 3\\ \cline{2-4}
			& \phantom{0}6 & 2 & 1\\ \cline{2-4}								
			& 14 & 1 & 1\\ 
			\hline 
		\end{tabular} 
	\end{center}
	\caption{The number of triple-intersection Latin rectangles (TILR) 
		with \mbox{$n=7$} sorted by autotopism group order.}
	\label{tlr}
\end{table}

In Table~\ref{tlr2} we give the number of such rectangles that are 
compatible with some proper double array. The maximum number of 
columns which are compatible with a double array is 2. Among the 
resulting non-isotopic double arrays for $(7,3,1)$ and $(7,4,2)$,
we see three different double arrays, when taking transposes into 
account. The rectangles in Figure~\ref{tda} are examples where the two 
compatible columns yield non-isotopic arrays, as indicated by subscripts.

\begin{table}[ht]
	\begin{center}
		\begin{tabular}{|c|c|c|c|}
			\hline 
			$(n,k,\lambda)$ & \# compatible TILR & \# compatible columns & \# DA \\
			\hline
			(7,3,1) & \phantom{00 000 00}6 & \phantom{00 000 00}8 & \phantom{00 0}2 \\ 
			\hline 	
			(7,4,2) & \phantom{00 000 0}97 & \phantom{00 000 }104 & \phantom{00 0}2 \\ 
			\hline 	 	
		\end{tabular} 
	\end{center}
	\caption{The number of triple-intersection Latin rectangles (TILR giving 
		proper double arrays.}
	\label{tlr2}
\end{table}

\begin{figure}
\begin{center}
\subfigure{
	\centering
	\begin{tabular}{|ccccccc|}
		\hline 
		0 & 1 & 2 & 3 & 4 & 5 & 6\\
		1 & 2 & 0 & 5 & 3 & 6 & 4\\
		2 & 3 & 4 & 6 & 0 & 1 & 5\\ 
		3 & 4 & 5 & 2 & 6 & 0 & 1\\
		\hline \hline
		  &   & D$_2$ & D$_1$ &   &   &  \\	
		\hline
	\end{tabular}
}
\subfigure{
\centering
		\begin{tabular}{|ccccccc|}
		\hline 
		0 & 1 & 2 & 3 & 4 & 5 & 6\\
		1 & 0 & 3 & 5 & 6 & 4 & 2\\
		2 & 3 & 4 & 6 & 0 & 1 & 5\\ 
		\hline \hline
		  &   &   & D$_1$ & D$_3$ &   &  \\	
		\hline
	\end{tabular}
}
\end{center}
\caption{Two examples of triple-intersection Latin rectangles with two 
	columns that 
	are compatible with non-isotopic proper double arrays. Subscripted D
	indicate the resulting non-isotopic double arrays, taking transposes into
	account.}
\label{tda}
\end{figure}

For triple-intersection Latin rectangles we have also found two examples 
which are 
compatible with proper sesqui arrays, as indicated in Table~\ref{tlr3}. 
We also found transposes of proper sesqui arrays in the case $4 \times 7$, 
as indicated in Table~\ref{tlr4}, and here the maximum number of compatible 
columns was three. We include all the resulting sesqui arrays 
here (in normalized form) in Figures~\ref{tsa} and \ref{tsa2}, 
since such arrays are scarce in the literature. We note that we only find
two non-isotopic sesqui arrays S$_1$ and S$_2$, when taking transposes into 
account, and that S$_1$ in fact recurs from Figure~\ref{YA7SAtb}.

\begin{table}[ht]
	\begin{center}
		\begin{tabular}{|c|c|c|c|}
			\hline 
			$(n,k,\lambda)$ & \# compatible TILR & \# compatible columns & \# SA   \\ 
			\hline 	
			(7,3,1) & \phantom{00 000 00}2 & \phantom{00 000 00}2 & \phantom{00 0}2 \\ 
			\hline 	
			(7,4,2) & \phantom{00 000 00}0 & \phantom{00 000 00}0 & \phantom{00 0}0 \\ 
			\hline 	 	
		\end{tabular} 
	\end{center}
	\caption{The number of triple-intersection Latin rectangles (TILR) giving 
		proper sesqui arrays.}
	\label{tlr3}
\end{table}

\begin{table}[ht]
	\begin{center}
		\begin{tabular}{|c|c|c|c|}
			\hline 
			$(n,k,\lambda)$ & \# compatible TILR & \# compatible columns & \# SA$^T$ \\ 
			\hline 	
			(7,3,1) & \phantom{00 000 00}0 & \phantom{00 000 00}0 & \phantom{00 0}0 \\ 
			\hline 	
			(7,4,2) & \phantom{00 000 0}73 & \phantom{00 000 0}78 & \phantom{00 0}2 \\ 
			\hline 	 	
		\end{tabular} 
	\end{center}
	\caption{The number of triple-intersection Latin rectangles giving 
		transposes of proper sesqui arrays.}
	\label{tlr4}
\end{table}

\begin{figure}
\begin{center}
\subfigure[S$_1$]{

	\begin{tabular}{|ccccccc|}
		\hline 
		0 & 1 & 2 & 3 & 4 & 5 & 6 \\ 
		1 & 2 & 0 & 4 & 5 & 6 & 3 \\ 
		2 & 3 & 4 & 5 & 6 & 1 & 0 \\ 
		\hline \hline
		  &   &   &   &   &   & S$_1$ \\	
		\hline
	\end{tabular}
	
	\vspace{12pt}
	
	\begin{tabular}{|cccc|}
		\hline 
		0 & 1 & 3 & 4 \\ 
		1 & 2 & 4 & 5 \\ 
		2 & 3 & 5 & 0 \\ 
		\hline 
	\end{tabular}
}
\hfill
\subfigure[S$_2$]{

	\begin{tabular}{|ccccccc|}
		\hline 
		0 & 1 & 2 & 3 & 4 & 5 & 6 \\ 
		1 & 2 & 4 & 0 & 5 & 6 & 3 \\ 
		2 & 3 & 1 & 4 & 6 & 0 & 5 \\
		\hline \hline
		  &   &   & S$_2$ &   &   &   \\	
		\hline
	\end{tabular}
	
	\vspace{12pt}
		
	\begin{tabular}{|cccc|}
		\hline 
		0 & 1 & 3 & 4 \\ 
		1 & 2 & 4 & 5 \\ 
		2 & 0 & 5 & 3 \\ 
		\hline 
	\end{tabular}

}
\end{center}
\caption{The triple-intersection Latin rectangles of size $3\times 7$ 
	that give proper sesqui arrays, together with the corresponding 
	sesqui arrays.}
\label{tsa}
\end{figure}

\begin{figure}
\begin{center}
\subfigure{

	\begin{tabular}{|ccccccc|}
		\hline 
		0 & 1 & 2 & 3 & 4 & 5 & 6 \\ 
		1 & 0 & 3 & 4 & 5 & 6 & 2 \\ 
		2 & 3 & 5 & 6 & 0 & 1 & 4 \\ 
		3 & 4 & 6 & 5 & 2 & 0 & 1 \\ 
		\hline \hline
		  &   &   & S$_1$ & S$_1$ & S$_2$ &   \\	
		\hline
	\end{tabular}
}
\hfill
\subfigure[S$_1^T$]{

	\begin{tabular}{|ccc|}
		\hline 
		0 & 1 & 4 \\ 
		1 & 4 & 2 \\ 
		2 & 3 & 5 \\ 
		3 & 5 & 0 \\ 
		\hline 
	\end{tabular}

}
\hfill
\subfigure[S$_2^T$]{
	\begin{tabular}{|ccc|}
		\hline 
		0 & 1 & 4 \\ 
		1 & 4 & 0 \\ 
		2 & 3 & 5 \\ 
		3 & 5 & 2 \\ 
		\hline 
	\end{tabular}

}
\end{center}
\caption{Example of a triple-intersection Latin rectangle of size $4\times 7$ 
	that gives transposes of proper sesqui arrays for three compatible
	columns, together with the corresponding non-isotopic transposed 
	sesqui arrays S$_1^T$ and S$_2^T$.}
\label{tsa2}
\end{figure}

%\CH{Array ~\ref{tsa}(a) yields the sesqui array
%\begin{tabular}{|cccc|}
%	\hline 
%	1 & 2 & 4 & 5 \\ 
%	0 & 1 & 3 & 4 \\ 
%	5 & 0 & 2 & 3 \\ 
%	\hline 
%\end{tabular} which is isotopic to \begin{tabular}{|cccc|}
%	\hline 
%	0 & 1 & 3 & 4 \\ 
%	1 & 2 & 4 & 5 \\ 
%	2 & 3 & 5 & 0 \\ 
%	\hline 
%\end{tabular}. The instance (b) yields the sesqui array \begin{tabular}{|cccc|}
%\hline 
%1 & 2 & 5 & 6 \\ 
%0 & 1 & 4 & 5 \\ 
%2 & 0 & 6 & 4 \\ 
%\hline 
%\end{tabular} which is isotopic to  \begin{tabular}{|cccc|}
%\hline 
%0 & 1 & 3 & 4 \\ 
%1 & 2 & 4 & 5 \\ 
%2 & 0 & 5 & 3 \\ 
%\hline 
%\end{tabular}}

%----------------------------------------------------------------------
%----------------------------------------------------------------------
\section{Concluding remarks}
\label{sec_concl}
%----------------------------------------------------------------------
%----------------------------------------------------------------------

With the computing time and storage available to us at present, we have 
exhausted the possibilities of complete enumeration of Youden rectangles.
A further line of inquiry might be to enumerate some restricted class
of Youden rectangles, satisfying some stronger conditions. Such conditions
would have to go beyond the structure of the symbol intersections between 
columns, since by only employing the balance condition, we can only 
distinguish between non-isotopic SBIBDs.

The new class of objects which we have named near Youden rectangles
(with only two column intersection sizes $\lambda_1$ and $\lambda_2$)
shows some promise with regard to two desirable properties. First, they
exist for far more parameter combinations than Youden rectangles. 
Second, they always have pairs of symbols covered either $\lambda_1$ or 
$\lambda_2$ times where $|\lambda_1 - \lambda_2| = 1$, so it may be 
expected that they perform reasonably well regarding statistical 
optimality. In a sense, they are as balanced as they can be. 
Investigating the statistical properties of these designs is beyond the
scope of this paper.

In relation to near Youden rectangles, we would like 
to pose the following question:
\begin{question}\label{q_PBYexist}
	For which combinations of $k$ and $n$ do near Youden rectangles exist?
\end{question}
As we noted earlier a result by Brown \cite{Brown88} implies that for 
$n=17, k=6$ a near Youden rectangle does not exist.

In relation to triple, double and sesqui arrays, we would like to pose the 
following questions:
\begin{question}\label{DAnotYR}
	For a given set of parameters, how many double arrays are there that 
	cannot be constructed from any Youden rectangle by removing a column 
	and all the symbols in that column,	and then exchanging the roles of 
	symbols and columns?
\end{question}

\begin{question}
	For a given set of parameters, can every double, triple, and 
	(transpose of) sesqui array be obtained from a Youden rectangle or 
	a triple-intersection Latin rectangle by Construction~\ref{RagNag}?		
\end{question}

Here one could of course extend the set of allowed intersection sizes in 
the Latin rectangle all the way up to $k$, so the focus is on whether a 
small span of intersection sizes suffices.

We hope to return to these questions in future work.

%----------------------------------------------------------------------
%----------------------------------------------------------------------
\section*{Acknowledgments}
%----------------------------------------------------------------------
%----------------------------------------------------------------------

The computational work was performed on resources provided by the 
Swedish National Infrastructure for Computing (SNIC) at 
High Performance Computing Center North (HPC2N).
This work was supported by the Swedish strategic research programme eSSENCE.    
This work was supported by The Swedish Research Council grant 2014-4897. 

The authors would like to thank the anonymous reviewers for their constructive criticism which improved the paper.

\end{document}